\documentclass[12pt,a4paper]{amsart}
\usepackage{amsmath,amssymb,amsthm,amscd,mathrsfs}
\usepackage[text={160mm,230mm},centering,headsep=2mm,footskip=5mm]{geometry}
\usepackage[dvipsnames]{xcolor}
\usepackage[colorlinks=true,linkcolor=RoyalBlue,citecolor=Mahogany,pagebackref]{hyperref}
\usepackage{graphicx}
\usepackage{tikz-cd}
\usepackage{tikz}
\usepackage{tikzsymbols}
\usepackage{mathtools}
\usetikzlibrary{patterns}
\usepackage{comment}
\usepackage[all]{xy}
\SelectTips{cm}{10}
\everyxy={<2.5em,0em>:} 
\usepackage{stmaryrd}

\newcounter{CountAlpha}

\theoremstyle{theorem}

\newtheorem{Thm}{Theorem}[section]
\newtheorem{Lem}[Thm]{Lemma}
\newtheorem{Cor}[Thm]{Corollary}

\theoremstyle{definition}
\newtheorem{Def}[Thm]{Definition}
\newtheorem{Ex}[Thm]{Example}
\newtheorem{Rk}[Thm]{Remark}

\newtheorem{Constr}[Thm]{Construction}

\newcommand{\car}{\operatorname{char}}

\newcommand{\conv}{\operatorname{conv}}

\newcommand{\Spec}{\operatorname{Spec}}

\newcommand{\ini}{\operatorname{in}}
\newcommand{\inv}{\operatorname{inv}}
\newcommand{\Int}{{\operatorname{Int}}}
\newcommand{\sep}{{\operatorname{sep}}}

\newcommand{\K}{k}

\newcommand{\gqz}{{\geq 0}}
\newcommand{\Bl}{{\operatorname{Bl}}}

\newcommand{\IA}{\mathbb{A}}

\newcommand{\IG}{\mathbb{G}}
\newcommand{\IP}{\mathbb{P}}
\newcommand{\IR}{\mathbb{R}}
\newcommand{\IQ}{\mathbb{Q}}
\newcommand{\IZ}{\mathbb{Z}}

\newcommand{\Gm}{\mathbb{G}_m}
\newcommand{\Gmu}{\pmb{\mu}}

\newcommand{\cC}{\mathcal{C}}

\newcommand{\cI}{\mathcal{I}}

\newcommand{\cO}{\mathcal{O}}

\newcommand{\cQ}{\mathcal{Q}}
\newcommand{\cS}{\mathcal{S}}
\newcommand{\cU}{\mathcal{U}}
\newcommand{\cW}{\mathcal{W}}

\newcommand{\cZ}{\mathcal{Z}}

\newcommand{\mm}{\mathfrak{m}}

\newcommand{\rd}{{\operatorname{red}}}

\numberwithin{equation}{section}

\newcommand{\double}{\genfrac..{0pt}1
{\raise -1pt\hbox{$\scriptstyle\longrightarrow$}}{\raise 3pt\hbox
{$\scriptstyle\longrightarrow$}}}

\begin{document}

\title{Torus actions, weighted blow-ups, and desingularization of plane curves}



\author{Dan Abramovich}
\address{\tiny Department of Mathematics, Brown University, Box 1917,
	Providence, RI~02912, USA}
\email{dan\_abramovich@brown.edu}

\author{Ming Hao Quek}
\address{\tiny Harvard University, Department of Mathematics, 1 Oxford Street, Cambridge, MA 02138}
\email{mhquek@math.harvard.edu}

\author{Bernd Schober}
\address{\tiny  None
(Hamburg, Germany)}
\email{schober.math@gmail.com}

\date{\today}

\begin{abstract}
	
	Given a singular 
	hypersurface in a regular
	2-dimensional scheme 
	essentially of finite type over a field, 
	we construct an embedded resolution of singularities by weighted blow-ups.
	This differs from our earlier work 
	which required multi-weighted blow-ups.
	We deduce an inductive argument, despite the fact that higher dimensional tangent spaces arise, by taking  torus actions and equivariant centers into account.
		In addition, we do not have to restrict to perfect base fields.

\end{abstract}

\maketitle

\setcounter{tocdepth}{1}
\tableofcontents

\section{Introduction}

In the present article, we are concerned with \emph{resolution of singularities} in positive characteristics. This is in general a challenging topic, however we focus on the case of plane curves,  known since the early days of the subject, beginning with Newton.

A new method of resolution in characteristic 0, using stack-theoretic weighted blow-ups, was introduced in \cite{ATW_weighted, Marzo, McQuillan}.
One must wonder if this method has implications in positive characteristics, how it relates to existing methods, and what possible hurdles it encounters. 
Specifically, colleagues have asked us to elaborate the first case, of plane curves.

Stack-theoretic weighted blow-ups stood in contrast with earlier methods, as traditional order reduction  necessitates, in general, numerous steps.
However, in positive characteristic, 
a weighted blow-up $ S' \to S $ of a regular surface $ S $
provides an Artin stack of the form $ S' = [\cW'/\IG_m]$, 
where the action has possibly inseparable stabilizers, and $\cW'$ is a 3-dimensional scheme.
Therefore, we are facing an inductive challenge
as we do not stay within the case of plane curves when performing weighted blow-ups.
In other words, even as we  prove that the order strictly decreases after the blow-up, 
we cannot apply an induction on the maximal value achieved by the order 
since the resulting singularity is not a plane curve.

In our earlier work \cite{AQS-plane}, 
we overcame this challenge 
under the additional assumption of a perfect base field
via replacing $S'$ by another stack-theoretic blow-up $S'' \to S$, which is a Deligne--Mumford stack, where the order is still reduced after blowing up, allowing us to repeat the process since we stay within the setting of plane curves.

In this paper, we provide  a complete treatment of weighted resolution of plane curves directly, 
using the weighted blow-up $S' = [\cW'/\IG_m]$ itself. 
This has the benefit that the \emph{logarithmic} order of the singularity drops at each step,
which implies that the order drops as well. 
In order to be able to repeat blowing up $S'$, we expand the generality where we allow higher dimensional schemes, but require them to possess torus actions. 
In contrast with \cite{AQS-plane}, 
our methods allow us to handle curves defined over an arbitrary field $k$, which need not be perfect.
The precise setup is as follows:

\begin{Thm}
	\label{Thm:Main2} 
	Let $ \cZ \subset \cW $ be a hypersurface in a regular scheme  $ \cW $ of dimension $n+2$ 
	which is essentially of finite type over an arbitrary field $ k $.
	Assume given an action of $T=\IG_m^n$ on $\cW$ stabilizing $\cZ$, having finite, possibly non-reduced, stabilizers.  
	Let $ \cQ \subset \cZ $ be a singular closed orbit
	such that the reduction $ \cZ_{\rm red} $ of $ \cZ $ is singular along $\cQ$. 
	\begin{enumerate} 
	\item\label{It:well-defined-eq} There is a well-defined, unique $T$-invariant center $J=(x_1^{a_1}, x_2^{a_2})$ of maximal invariant $(a_1,a_2)$ that is admissible for $ \cZ \subset \cW$ and supported along $\cQ$,
	where $ a_1 $ is the order of $ \cZ$ along $ \cQ $. 
	Moreover, $J$ is stable under base change to separable field extensions of $k$.

	\item\label{It:order-drops-eq} Let $\bar J = (x_1^{1/w_1}, x_2^{1/w_2})$ be the associated reduced center. Let $B \to \cW$ be the associated degeneration to the normal cone of $\bar J$, with a natural open subscheme $\cW':=B_+ \subset B$ admitting a $T' = T \times \IG_m$-action lifting the $T$-action on $\cW$. Then the action of $T'$ on $\cW'$ has finite stabilizers, and the morphism  $[\cW' /T'] \to [\cW /T]$ is the weighted blow-up along $\bar J$, which is a proper and birational morphism of tame Artin stacks.
	
	\item\label{It:order-drops-new} Let $\cZ'\subset \cW'$ be the proper transform of $\cZ$. 
	Then the logarithmic order of $\cZ'$ with respect to the exceptional divisor at every point lying over $ \cQ $ is strictly smaller than $ a_1 $.
	In particular, the invariant  defined  in Part \eqref{It:well-defined-eq} of the center  associated to $\cZ' \subset \cW'$  at every point of $\cW'$  lying over $ \cQ$ is strictly smaller than $ (a_1,a_2)$.

	\end{enumerate}
		Iterating $\cW' \to \cW$, one obtains in finitely many steps $\cW^{(m)} \to \cW^{(m-1)}$,  with the reduction $\cZ^{(m)}_\rd \subset \cW^{(m)}$ nonsingular.
		Taking quotients  $[\cZ^{(m)}/T^{(m)}]_\rd \subset [\cW^{(m)}/T^{(m)}]$ we obtain   an embedded resolution of singularities of $[\cZ/T]\subset [\cW/T]$ by tame Artin stacks. 
\end{Thm}

Theorem \ref{Thm:Main2} is proved in Section \ref{Sec:Proofs}. Recall that a scheme $S$ is essentially of finite type over a field $ k $ if $S = \cup_{i=1}^m \Spec (A_i)$, where $A_i$ are localizations of finite-type $ k $-algebras.

In previous papers \cite{ATW_weighted,AQ,Quek-Rydh}, there are multiple ways to define centers. Among these, two are used here: \emph{monomial valuations} and \emph{Rees algebras}. In Section~\ref{Sec:inv}, we chose to treat centers as monomial valuations as these are sufficient for defining unique invariants and unique centers. In Section~\ref{Sec:weighted}, we use the description of centers in terms of Rees algebras to explain weighted blow-ups.

An immediate consequence  of Theorem~\ref{Thm:Main2} 
 is the following:

\begin{Thm} \label{Thm:resolution} 
Let $C \subset S$ be a hypersurface in a regular $ 2 $-dimensional scheme $S$ which is essentially of finite type over a field $ k $.
 Then the sequence of weighted blow-ups of Artin stacks introduced in Theorem~\ref{Thm:Main2}, applied to   $\cZ=C$ and $\cW = S$ and $T$ the trivial group,  
gives   a stack-theoretic  embedded resolution of singularities of $C \subset S$.

	\end{Thm}

An application of Bergh's destackification \cite{Bergh} provides
 a scheme-theoretic embedded  resolution
 of $\cZ \subset \cW$ and of  $[\cZ/T] \subset [\cW/T]$, see \cite[Theorem~8.1.3]{ATW_weighted}. For further discussion see Section \ref{Sec:destack}.

Rather than proving the result directly on higher dimensional schemes, we pass to the localization at the generic point of $\cQ$ to reduce to the case of a singular 
hypersurface
 $C\subset S$ on a regular 2-dimensional scheme. This approach requires the following upgrade of \cite[Theorem 1.1]{AQS-plane} to the case where $S$ is essentially of finite type over an arbitrary field:

\begin{Thm}
	\label{Thm:Main1} 
	Let $ C \subset S $ be a hypersurface in a regular $ 2 $-dimensional scheme $ S $ essentially of finite type over a field $k$.
	Let $ q \in C $ be a singular closed point
	such that the reduction $ C_{\rm red} $ of $ C $ is singular at $ q $.
	\begin{enumerate} 
	\item\label{It:well-defined}  There is a well-defined, unique center $J=(x_1^{a_1}, x_2^{a_2})$ of maximal invariant $(a_1,a_2)$ admissible for $ C \subset S $ at $q$,
	where $ a_1 $ is the order of $ C$ at $ q $. 
	Moreover, $J$ is stable under base change to separable field extensions of $k$. (See Theorem \ref{Thm:J_unique}.)
	
	\item\label{It:birational} 
	Let $\bar J = (x_1^{1/w_1}, x_2^{1/w_2})$ be the associated reduced center. Its stack-theoretic weighted blow-up $S' = Bl_{\bar J}(S)\to S$ is a proper birational morphism from a smooth Artin stack $S'$ which is a global quotient of a variety by $\IG_m$, in particular tame.
	(See \cite[Lemma~1.1.2 and Remark~3.2.10]{Quek-Rydh}.) 
	
	\item\label{It:order-drops} The logarithmic order of the proper transform of $C$  with respect to the exceptional divisor  at every point of the blow-up lying over $ q $ is strictly smaller than $ a_1 $. (See Section \ref{Sec:Proofs}.)
	\end{enumerate}
\end{Thm}

In Section~\ref{Sec:comparison}  we compare the resolution via weighted blow-ups presented here with  other methods for resolution of plane curve singularities.

\begin{Rk}
The assumption here that $S$ be essentially of finite type over a field is minimalistic --- it is sufficient for our method to work. There is no doubt that one can further relax this assumption, but we opt to keep things simple in this paper.

Also for simplicity we do not pursue logarithmic resolution, though the same methods apply.
\end{Rk}
\begin{Rk}
One could restrict attention to the case where the subscheme $C$ in Theorems~\ref{Thm:resolution},  \ref{Thm:Main1} and correspondingly $\cZ$ in Theorem \ref{Thm:Main2} is reduced. However the proof of Theorem \ref{Thm:Main1}  requires looking at the weighted normal cone of $C$, which is often non-reduced even when $C$ is. 
The generality we consider, with the reduction being singular, requires no additional effort.

This generality is also in line with model treatments of resolution of non-reduced schemes, such as \cite[Definition~6.2]{CJS}: a non-reduced scheme $X$ is \emph{quasi-regular} if $X_\rd$ is regular and $X$ normally flat along $X_\rd$, namely the Hilbert--Samuel function is constant along $X_\rd$
(see \cite{CPS}). 
A quasi-regular scheme is as resolved as a non-reduced scheme can possibly be.

But for a hypersurface, the Hilbert--Samuel function is automatically constant once $X_\rd$ is regular. See  also \cite[Remarks 6.19, 6.20]{CJS} and related discussion in \cite[Remark~(1.13)]{HiroCharPoly}. Thus our singularity assumption is exactly  that the subscheme $C$ or $\cZ$ is not quasi-regular.

\end{Rk}

\subsection*{Acknowledgements}

This research work was supported in part by funds from BSF grant 2022230, NSF grant DMS-2401358, and Simons Foundation Fellowship SFI-MPS-SFM-00006274.
Dan Abramovich thanks the Hebrew University of Jerusalem  for its hospitality during Spring 2025.

We thank Andr\'e Belotto da Silva, Ang\'elica Benito, Raymond van Bommel, Ana Bravo, Vincent Cossart, Santiago Encinas, Mike Montoro, James Myer,  Michael Temkin and Jaros{\l}aw W{\l}odarczyk for discussions on the topic. We thank the referee for numerous comments that helped us improve the paper.

\section{A brief introduction to Hironaka's characteristic polyhedron}
\label{Sec:Hiro}

First, we give a short overview on Hironaka's characteristic polyhedron \cite{HiroCharPoly} in the special case of particular principal ideals which is the situation that we are interested in. 
Detailed introductions to the topic can be found in \cite[Chapter~8 and appendix]{CJS} as well as in \cite{CS-Compl}, for instance.

Let $ ( \cO, \mm, K ) $ be a regular local ring
with maximal ideal $ \mm $ and residue field $ K \simeq \cO/\mm $.
Let $f \in \mm $ be a non-zero element.
Recall that the order of $ f $ at $ \mm $ is defined as 
\[
	\nu := \nu_\mm(f) := \sup \{ \alpha \in \IZ_+ \mid f \in \mm^\alpha  \} .
\]

We make the assumption on $ f $ that there exists a regular system of parameters $ (x,y) = (x_1, \ldots, x_e,y) $ of $\cO $ such that
\begin{equation}
	\label{eq:Dir=1}
		f \equiv Y^\nu \mod \mm^{\nu+1}
\end{equation}
where $ Y := y \mod \mm^2 $ denotes the image of $ y $ in the graded ring $ \operatorname{gr}_\mm(\cO) = \bigoplus_{a \geq 0} \mm^a/\mm^{a+1} $.
In Remark~\ref{Rk:dir_dim_zero} we address how we reduce to this situation for $ 1 $-dimensional schemes essentially of finite type over a field. 
Since $ \cO $ is Noetherian, we have a {\em finite} expansion 
\[
	f = y^\nu + \sum_{(A,b)} c_{A,b} x^A y^b 
\] 
with $ c_{A,b} \in \cO^\times \cup \{ 0 \} $ and multi-index notation $ x^A = x_1^{A_1} \cdots x_e^{A_e} $,
see \cite[Proposition~2.1]{Cossart-Pitant}.

\begin{Def}
	\begin{enumerate}
		\item 
		The {\em projected polyhedron} of $ f $ with respect to $ (x;y) $ is defined as
		\[
		\Delta(f;x;y) \ := \
		\conv \left( \left\{ \frac{A}{\nu - b} + \IR_\gqz^e  \ \Big| \  c_{A,b} \neq 0 \right\} \right) , 
		\]
		where $ \conv (*) $ denotes the convex hull.
		
		\item 
		For a vertex $ v  \in \Delta(f;x;y) $
		the {\em initial form of $ f $ at $ v $}
		is given by
		\[
		\ini_v (f) \ := \  Y^\nu + \sum_{(A,b) : \frac{A}{\nu- b} = v} \overline{c_{A,b}} X^A Y^b \ \in \  K[X,Y],
		\]
		where 
		$ \overline{c_{A,b}} $ are the images of the coefficients in the residue field $ K \simeq \cO/\mm $.

		\item 
		A vertex $ v = (v_1, \ldots, v_e) \in \Delta (f;x;y) $ is called {\em solvable} if $ v \in \IZ_{\geq 0}^ e $ and there exists a (unique) $ \lambda \in K \smallsetminus \{ 0 \} $ such that
		$ \ini_v (f) = (Y-\lambda X^v)^\nu  = (Y-\lambda X_1^{v_1}\cdots X_e^{v_e})^\nu $.
	\end{enumerate}
\end{Def}

Illustrative pictures of the projected polyhedron and the initial form at a vertex can be found in \cite[Examples 18.18, 18.23, and 18.31]{CJS}.

Suppose $ v $ is a solvable vertex. 
Let $ \epsilon \in \cO^\times $ be any lift of $ \lambda $ in $ \cO $, i.e., $ \epsilon \equiv \lambda \mod \mm $.  
If we set $ y' := y - \epsilon x^v $, then $ v \notin \Delta(f;x;y') $ and $  \Delta(f;x;y) \supsetneq  \Delta(f;x;y') $
by \cite[Lemma~(3.10)]{HiroCharPoly}.

\begin{Ex}
	Let $ \cO = K [[x,y]] $ be of dimension two for a field $ K $ of characteristic 2 and consider the polynomial $ f = y^2 + \lambda x^4  + x^5$
	with $ \lambda \in K $.
	The vertex of the projected polyhedron $ \Delta(f;x;y) $ is $ v = 2 $.
	
	If $ \lambda \in K^2 $ is a square, then $ v $ is solvable via the change $ y' := y + \mu x^2 $ with $ \mu \in K $ such that $ \mu^2 = \lambda $.
	The resulting polyhedron $ \Delta(f;x;y') $ has vertex $ v' = 5/2 $ which is not solvable.
	
	On the other hand, if $ K $ is imperfect and if $ [K^2 (\lambda):K^2] = 2 $, i.e., $ \lambda $ is not a square, then $ v = 2 $ is not solvable. 
\end{Ex}

For a non-empty convex subset $ \Delta \subseteq \IR_\gqz^e $, we define the number 
\[  
	\delta(\Delta) := \inf\{  v_1 + \cdots + v_e \mid (v_1, \ldots, v_e ) \in \Delta\} . 
\]
Furthermore we set $ \delta(\emptyset) := \infty $ 
and abbreviate $ \delta(f;x;y) :=  \delta(\Delta(f;x;y)) $.

Let us point out that $ \delta(f;x;y) = \infty $ if and only if $ f =\epsilon  y^\nu  $ for some unit $ \epsilon\in \cO $.
In this case the reduction of $ \{ f = 0 \} $ is regular.

\begin{Thm}
	\label{Thm:poly}
	Let $ f \in \mm $ be an element such that \eqref{eq:Dir=1} holds.
	\begin{enumerate}
	\item \label{It:char_poly}
	There is an element $ z \in \cO $ such that $ (x,z) $ is a regular system of parameters for $ \cO $ 
	and $ \Delta(f;x;z) $ is minimal with respect to inclusion among all possible $ \Delta(f;x;y) $ with $ (x,y) $ a regular system of parameters for $ \cO $.
	
	\item \label{It:delta_inv}
	The number $ \delta(f;x;z) $ 
	is an invariant of the singularity $ \cO/(f) $.
	\end{enumerate}
\end{Thm}

The minimal polyhedron $ \Delta(f;x) := \Delta(f;x;z) $ 
appearing in Theorem~\ref{Thm:poly}\eqref{It:char_poly}
is called {\em the characteristic polyhedron of $ f $}.
We emphasize that this is an ad-hoc definition. 
Originally, $ \Delta(f;x) $ is defined as the intersection over all possible projected polyhedra $ \Delta(f;x;y) $ with varying $ y $.

\begin{proof}[Proof of Theorem~\ref{Thm:poly}] 
	This follows from a combination of results from \cite{HiroCharPoly,Cossart-Pitant-compl,CJSch}:
Hironaka proved that in the $ \mm $-adic completion $ \widehat{\cO} $ there is an element $ \widehat{z} $ with $ \Delta(f;x;\widehat{z}) = \Delta(f;x) $ in \cite[Theorems~(3.17) and (4.8)]{HiroCharPoly}.
In \cite{Cossart-Pitant-compl} the existence of $ z \in \cO $, not necessitating completion, was shown for the given setting.
Generalizations followed in \cite{CS-Compl}.  
The second part on the invariant is proven in \cite[Corollary~B(3)]{CJSch}.
\end{proof}

\begin{Def}
	\label{Def:Fdelta}
	Let $ f \in \cO $ be as above and let $ z $ be as in Theorem~\ref{Thm:poly}. 
	Set $ \delta := \delta(f;x;z)  $.
	The {\em $ \delta $-initial form} of $ f  = \sum_{(A,b)}  c_{A,b} x^A z^b $ is defined as 
	the initial form of $ f $ at the vertex $ \delta $, which can be rewritten as
	\[
	\ini_\delta(f)
	\ = \  
	\sum_{(A,b) \, : \, \frac{|A|}{\delta\nu} + \frac{b}{\nu} = 1} \overline{c_{A,b}} X^A Z^b.
	\]

	We introduce the following lift to $ \cO $ of the $ \delta $-initial form of $ f $:
	\[
		F_\delta \ :=  \  \sum_{(A,b) \, : \, \frac{|A|}{\delta\nu} + \frac{b}{\nu} = 1} c_{A,b} x^A z^b \in \cO.
	\]
\end{Def}

Observe that the definition of the lift depends on the choice of lifts for the coefficients of the initial forms from $ K = \cO/\mm $ to $ \cO $.
Since they  originate from certain elements $ c_{A,b} \in \cO^\times \cup \{ 0\} $, we choose them as the lifts.

\begin{Ex}
	\phantomsection
	\label{Ex:1_2_3}
	\begin{enumerate}
		\item \label{Ex:1_imperfect}
		Let $ \K $ be an imperfect field of characteristic $ p $ 
		and $ a \in  \K \smallsetminus \K^p $.
		Consider 
		\[  
			f \, = \, y^p- x_1^m (x_1^p + a)^{pm+1} - a x_2^{pm}
		 \ \in \ k[x_1,x_2,y] 
		\]
		with $ m > p  $ coprime to $ p $.
		
		At the closed point $ q $ given by the maximal ideal $ (x_1, x_2, y) $ 
		the residue field is equal to $ k $.
		The vertices of $ \Delta(f;x;y) $ are 
		$ v_1 := (\frac{m}{p}, 0) $ and $ v_2 := (0, m) $.
		Observe that $ \ini_{v_2} (f) = Y^p - a X_2^{pm} $ and since $ a \notin \K^p $, the vertex $ v_2 $ is not solvable. 
		Thus, none of the vertices is solvable and $  \Delta(f;x;y) =  \Delta(f;x) $ is the characteristic polyhedron (at $ q $).
		
		Let us look at the point $ \widetilde{q} $ corresponding to the maximal ideal $ (x_1^p + a , x_2, y) $.
		The residue field at this point is $ K := k [x_1]_{(x_1^p+a)}/(x_1^p +a)  $. 
		Possible local parameters at $ \widetilde{q} $ 
		are $ ( u_1 := x_1^p + a , x_2, y ) $.
		By substituting $ a = u_1 - x_1^p $,
		$ f $ can be rewritten as 
		\[
			f \,  = \, y^p- x_1^m u_1^{pm+1} - u_1 x_2^{pm} + x_1^p x_2^{pm} .
		\]
		Taking into account that $ x_1 $ is invertible at $ \widetilde q $,
		the vertices of the corresponding projected polyhedron $ \Delta(f;u_1, x_2;y) $ are 
		$ \widetilde{v_1} := (m+\frac{1}{p},0) $ and $ \widetilde{v_2} := (0, m)$.
		This time $ \widetilde{v_2} $ is solvable since $ \ini_{\widetilde{v_2}} (f) = Y^p + \overline{x_1}^p X_2^{pm} = (Y + \overline{x_1} X_2^m )^p $
		where  $ \overline{x_1} $  
		denotes the image of $ x_1 $ in the residue field $ K $.
		Introducing  $ z := y + x_1 x_2^{m} $
		the vertices of $ \Delta(f;u_1, x_2;  z) $
		are $ \widetilde{v_1} $ and $ \widetilde{v_3} := (\frac{1}{p},m) $.  
		Both vertices are not solvable and hence $ \Delta(f;u_1, x_2; z) = \Delta(f;u_1, x_2) $ is the characteristic polyhedron (at $ \widetilde q $).

		\item \label{Ex:2_localize}
		Let $ \K $ be a perfect field of positive characteristic $ p $
		and let $ \cO $ be the localization of $ \K[x,y,z] $ at the ideal $ (x,y) $. 
		Let $ m \in \IZ_+ $ with $ m > p $
		and 
		 \[ 
		 	f \, = \, y^{2p} + x^m y^p z^{m+p} + x^{2m} z^{p+1} + x^{2m-1} y  \ \in \ \cO .
		 \]
		Notice that $ (x,y) $ is a regular system of parameters for $ \cO $ and $ z $ is invertible in $ \cO $.
		The projected polyhedron $ \Delta(f;x;y) = \Delta (f;x)$ coincides with the characteristic polyhedron since the vertex is $ \delta = \frac{m}{p} $ which is not solvable. 
		
		We have $ \ini_\delta(f) = Y^{2p} + X^m Y^p \overline{z}^{m+p} + X^{2m} \overline{z}^{p+1} $,
		where $ \overline{z} $ denotes the image of $ z $ in the residue field $ \cO/ (x,y) $. 
		
	\end{enumerate}
\end{Ex}

\section{Invariant, parameters, and center}\label{Sec:inv}

	In the next step, we define the invariant in the setting of Theorem~\ref{Thm:Main1} and deduce from this the designated center for the weighted blow-up.
	
	Let $ C \subset S $ be a hypersurface in a regular $ 2 $-dimensional scheme $ S $ which is essentially of finite type over a field $ k $
		and let $ q \in C $ be a singular closed point such that $ C_{\rm red} $ is singular at $ q $. 
	Let $ (\cO,\mm,K) $ be the local ring of $ C $ at $ q $
	and let $ f \in \mm $ be an element defining $ C $ locally at $ q $.
	Sometimes, we will abuse notation by writing $ C = V(f) $. 

\begin{Rk}
	\label{Rk:dir_dim_zero}
	Let us address the situation if \eqref{eq:Dir=1} does not hold. 
	Then $ \ini_\mm(f) := f \mod \mm^{\nu+1} \in \operatorname{gr}_\mm(\cO) $ cannot be expressed as polynomial in a single variable.
	For readers familiar with Hironaka's notion of the directrix (for example, see~\cite[Definitions~2.8 and 2.26]{CJS}), the latter means that the dimension of the directrix is zero.
	Due to a result of Hironaka \cite[Theorem~2]{HiroAdditive} which was generalized by Mizutani \cite{Mizutani} (see also \cite[Theorem~3.14]{CJS})
	blowing up the closed point $ q $ will lead in this situation to a strict decrease of the order at points lying above $ q $.

	In the remainder of the article 
	we include the discussed case by setting $ \delta := 1 $ and $ \ini_\delta (f) := \ini_\mm (f) $ if  
	\eqref{eq:Dir=1} does not hold.
\end{Rk}

\begin{Def}\label{Def:invariant}
	Let $ q \in C \subset S $ be as before, locally given as $ q \in V(f) 
	\subset \Spec(\cO) $. 
		The {\em invariant of $ C $ at $ q $} is defined as
		\[ 
			\inv_{C} ( q ) \, := \,  \inv_{f} (q) \, := \,  (a_1,a_2) \, := \,  (\nu, \delta\nu) \ \, \in \ \,  \Gamma \, := \,  \IZ_{\geq 0} \times \frac{1}{a_1!} \IZ_{\geq 0} ,
		\] 
		where $ \nu $ is the order of $ C $ at $ q $, $ \delta \in \IQ_+ $ is the number introduced in Theorem~\ref{Thm:poly}\eqref{It:delta_inv} (resp.~Remark~\ref{Rk:dir_dim_zero}),
			and $ \Gamma $ is equipped with the lexicographical ordering.
\end{Def}

Recall that our assumption, $ C_{\rm red} $ is singular at $ q $, implies that $ \delta := \delta(f;x;z) < \infty $, see the paragraph just before Theorem~\ref{Thm:poly}.

Using parameters $ (z,x) $ with the property $ \Delta(f;x;z) = \Delta(f;x)$, we shall first determine our designated center for the weighted blow-up in terms of monomial valuations. In the next section, we will re-interpret the center as a Rees algebra (see Remark~\ref{Rk:centers}(3)), and discuss the weighted blow-up along the center. While the following definition depends on choices, we show in Theorem~\ref{Thm:J_unique} that the center is unique.
Despite working in a more general setup, the proof follows the same lines as \cite[Theorem~4.2]{AQS-plane}.

\begin{Def} \label{Def:centers}
	Let $ ( \cO, \mm, K ) $ be a regular local ring of dimension $ 2 $ which is essentially of finite type over a field $ k $.
		\begin{enumerate}
			\item 
			For {\em any} 
 parameters $ (x_1, x_2) $ of $ \cO $
 and {\em any} positive rational numbers $ a_1, a_2 \in \IQ_+ $, 
 			we define $J = (x_1^{a_1},x_2^{a_2})$ to be the monomial valuation 
 			$v_J\colon \cO \smallsetminus \{0\} \to \IQ$
 			given by 
 			\[
 				v_J(x_1) = \frac{1}{a_1},
 				\quad 
 				v_J(x_2) = \frac{1}{a_2}.
 			\]
 			We call such monomial valuation $J$ \emph{a center}.
 		\item 
 		For a non-zero element $ f \in \cO  \smallsetminus \{0\}$,
 		we say $J = (x_1^{a_1},x_2^{a_2})$ is {\em admissible} for $ f $
 		if we have $ v_J(f) \geq 1 $. 
		\end{enumerate}
\end{Def}

\begin{Rk}\label{Rk:centers}
	\phantomsection
	\label{Rk:v_J}
	\begin{enumerate}
		\item 
		Since $ v_J $ is a monomial valuation, we have:
		For any $ g \in \cO\smallsetminus \{0\} $
		with expansion $ g (x_1,x_2) = \sum\limits_{(i_1,i_2)} d_{i_1, i_2} x_1^{i_1} x_2^{i_2} $ for $ d_{i_1,i_2}\in \cO^\times \cup \{ 0 \} $,
		the following holds: 
		\[ 
		v_J(g) = \min\left\{ \frac{i_1}{a_1} +  \frac{i_2}{a_2} \ \Big| \  d_{i_1,i_2} \neq 0\right\}.
		\]
		In particular, we have $ v_J(x_1^{a_1}) = v_J(x_2^{a_2}) = 1 $. 
		
		\item 
		Observe that $ J =  (x_1^{a_1},x_2^{a_2}) $ is admissible for a given $ f $ if and only if $f\in (x_1^{a_1},x_2^{a_2})^{\Int}$.
		
		\item \label{It:local} 
		For resolution of $ 1 $-dimensional schemes the local perspective via valuations is enough. 
		In general, for instance in Theorem~\ref{Thm:Main2}, one has to describe centers globally, 
		as for example in \cite[Section~2.2]{ATW_weighted} via \emph{valuative $\IQ$-ideals}, or equivalently in \cite[Corollary 2.11]{MingHao} via integrally closed Rees algebras. See Remark~\ref{Rk:weighted-blow-ups} for a global description of centers as Rees algebras.
	\end{enumerate}
\end{Rk}

\begin{Thm}
	\label{Thm:J_unique}
	Let $ q \in C = V(f) 
	\subset \Spec(\cO) $ be as described at the beginning of this section.
	Let $ (a_1, a_2 ) := (\nu, \nu \delta ) $ 
and let $ (x_1, x_2) $ be a regular system of parameters for $ \cO $ such that $ \Delta (f;x_2;x_1) = \Delta(f;x_2) = [\delta ; \infty [  $.
Then $ J = (x_1^{a_1}, x_2^{a_2} ) $ is the unique admissible center for $ f $ for which $ (a_1,a_2) $ with $ a_1 \leq a_2 $ attains its maximum with respect to the lexicographical order.
Moreover, $J$ is stable under base change to separable field extensions of $ k $.
\end{Thm}

We note that this theorem implies Theorem~\ref{Thm:Main1}\eqref{It:well-defined}.

\begin{proof}
	Our goal is to show
	\begin{equation}
	\label{eq:lexmax} 
	\max_{\geq_{\rm lex}} \{ (b_1, b_2) \in \IQ_+^2 \mid \exists \, (y_1,y_2) :  (y_1^{b_1},y_2^{b_2}) \mbox{ admissible center for } f  \}
	= (a_1, a_2) .
	\end{equation}

	First, the left hand side of Equation~\eqref{eq:lexmax} is $ \geq (a_1, a_2) $
	since $ J $ is admissible by the definition of  the characteristic polyhedron.

Secondly, $ a_1 $ is unique:
If $ (y_1^{b_1},y_2^{b_2})$ is another admissible center with $b_1 \leq b_2$, 
we must have $ b_1 \leq a_1 $ since $ a_1 = \nu_\mm(f) $ which implies the existence of a monomial $ y_1^a y_2^{\nu-a} $ with coefficient a unit in the expansion of $ f $.

		Thirdly, $ a_2 $ is unique:
		Let $ (y_1^{\nu},y_2^{b_2})$ be an admissible center. 
		Using the definition of the characteristic polyhedron and $ \delta(f;y_2;y_1) \leq \delta $, 
		it follows that we have to have $ b_2 \leq \nu \delta = a_2 $ which implies Equation~\eqref{eq:lexmax}.

	Finally, we show that the center is unique: 
	Let $ J_x = ( x_1^{a_1} , x_2^{a_2} ) $
	and $ J_y = ( y_1^{a_1} , y_2^{a_2} ) $ be two admissible centers for $ f $. 
	If $ a_1 = a_2 $ then $ v_{J_x} = v_{J_y} = a_1 \nu_m $ 
	and the desired equality $ J_x = J_y $ holds. 
	Hence, suppose $ a_1 < a_2 $. 
	
	\begin{enumerate}
		\item [(a)]
		\underline{$ x_2 = y_2 $\,:} 
		By construction, we may vary $ x_2 $ for the given $ x_1 $ as long as we have $ \mm = (x_1, \widetilde{x_2}) $ for the different choice. 
		The analogous statement is true for $ y_2 $. 
		Since $\IP(\mm/\mm^2)$ has at least 3 points,
		we may assume without loss of generality that $ x_2 = y_2 $. 
		Further, this provides $ v_{J_y} (x_2) =v_{J_x} (x_2) = \frac{1}{a_2} $. 
		
		\item[(b)]
		 \underline{$ v_{J_y} (x_1) = v_{J_x} (y_1) = \frac{1}{a_1} $\,:}
		 If $ x_1 $ and $ y_1 $ differ by a unit factor, the assertion clearly holds. 
		 Therefore, consider an expansion of $ x_1 $ in the regular system of parameter $ (y_1, x_2) $,
		 say $
		 	x_1 = \epsilon y_1 + \mu x_2^A 
		 	$
		 	for  
		 	$ A \in \IZ_+, \epsilon, \mu \in \cO^\times $.
		 Since $ v_{J_y} $ is a monomial valuation, this provides $ v_{J_y} (x_1) \leq v_{J_y} (y_1) = \frac{1}{a_1} $.
		 
		 Assume the inequality is strict, $ v_{J_y} (x_1) < \frac{1}{a_1} $.   
		 Expand $ f = \sum_{ ( i_1, i_2 )} d_{i_1,i_2} x_1^{i_1} x_2^{i_2} $.
		 Since $ J_x $ is admissible for $ f $,
		 we have $ \frac{i_2}{a_2} \geq \frac{a_1-i_1}{a_1} $ for every $ (i_2, i_2) $ such that $  d_{i_1,i_2} \in \cO^\times $ is a unit. 
		 Moreover, we have an equality if and only if $ i_1 = a_1, i_2 = 0 $. 
		 Using the non-Archimedean property, it follows that $ v_{J_y} (f) = a_1 v_{J_y} (x_1) < 1 $.
		 This contradicts the assumption that $ J_y $ is an admissible center for $ f $
		 and therefore we obtain $ v_{J_y} (x_1) = \frac{1}{a_1} $.
		Applying the analogous arguments leads to  $ v_{J_x} (y_1) = \frac{1}{a_1} $.
		
		\item[(c)]
		\underline{$ J_x = J_y $\,:}
		This follows since $(x_1^{ m a_1},x_2^{m a_2})^\Int =  (y_1^{m a_1},y_2^{m a_2})^\Int$
		for every $ m \in \IZ_{\geq 0} $. 
	\end{enumerate}
 
	It remains to prove that $J$ is stable under base change to separable field extensions of $ k $.
	Let $ k' / k $ be a separable extension. 
	Using \cite[Prop.~2.15]{Cossart-Pitant}
	(applied for $ h := \ini_\delta (f) $ considered as element in $ S[X_1] $ for $ S := K[X_2]_{ (X_2) } $ and $ \widetilde S := S \otimes_{k} k' $)
	it can be deduced that the characteristic polyhedron does not change under the base change $ k'/k $ and thus $ J $ remains stable.

	Since our given situation is simpler, we provide a direct argument:
	We have to study the behavior of $ \phi := \ini_\delta (f) \in K[X_1,X_2] $ under separable extensions. 
	Let $ \phi_s $ be the element $ \phi $ considered as element in $ (K \otimes_{k} k') [X_1,X_2] $. 
	We have to show that $ \phi_s $ is not a $ \nu $-th power, $ \phi_s \neq (X_1 + \alpha X_2^\delta )^\nu $ for $ \alpha \in (K \otimes_{k} k' ) \smallsetminus \{0\} $.
	If $ \delta \notin \IZ_+$, this is certainly not possible.
	Hence let $ \delta \in \IZ_+ $ and assume $ \phi_s = (X_1 + \alpha X_2^\delta )^\nu $.
	Write $ \nu = p^m \nu' $ for $ \nu' $ prime to $ p $ if $ \car(k) = p > 0 $ and $ \nu= \nu' , p^m = 1 $ else.
	Since the extension is separable, we have to have $ \phi = \psi^{p^m} $ for some $ \psi \in K [X_1,X_2] $.   
	Clearly, $ \nu' \geq 2 $ in this case and $ \psi_s = (X_1 +\alpha X_2^\delta )^{\nu'} $.
	We have $ \psi = X_1^{\nu'} + \sum_{i=1}^{\nu'} c_i X_1^{\nu'-i} X_2^{\delta i} $ for $ c_i \in K $.
	Let $ Y_1 := X_1 + \frac{c_1}{\nu'} X_2^\delta $.  
	Substituting $ X_1 = Y_1 - \frac{c_1}{\nu'} X_2^\delta $ provides  
	$ \psi = Y_1^{\nu'} + \sum_{i=2}^{\nu'} c_i' Y_1^{\nu'-i} X_2^{\delta i} \neq Y_1^{\nu'}  $ for certain $ c_i' \in K $.
	Observe that $ J $ remains stable under this substitution and that $ c_1' = 0 $.
	The last condition is stable by the base change with $ k' $, but this contradicts $ \psi_s = (X_1 +\alpha X_2^\delta )^{\nu'} $. 
\end{proof}

In order to draw a connection to the weighted normal cone, we need to compare $ f $ and $ F_\delta $ (Definition~\ref{Def:Fdelta}).
	Note that $ F_\delta $ is the sum over all monomials of valuation $ 1 $ in an expansion of $ f $.
To emphasize the relation to the valuation $ v_J $, we introduce the notation $ f_{J,1} := F_\delta .$
Hence, we have
\begin{equation}
	\label{eq:expansion} 
	f  =  f_{J,1} + \widetilde h,
	\quad
	\mbox{ where } \widetilde h \in \cO 
	\mbox{ with } v_{J} (\widetilde h) > 1 .
\end{equation}

\begin{Cor}\label{Cor:leading-term} 
	The reduction of $ V( f_{J,1}) $ 
	is singular at $ q $ and $\inv_{f_{J,1}}(q) = (a_1,a_2)$ with the same center  $(x_1^{a_1},x_2^{a_2})$ as for $ f $.
\end{Cor}
\begin{proof}
		First, we have  $ \ini_\delta (f_{J,1})  = \ini_\delta (f) $.
		This implies that the vertex
		$ \delta \in \Delta (f_{J,1};x_2;x_1) $ 
		is not solvable. 
		In particular, $ f_{J,1} $ is not equal to the power of a regular parameter
		which implies the claim on the reduction of $ V( f_{J,1}) $. 
		
		Furthermore, the characteristic polyhedra of $ f_{J,1} $ and $ f  $ coincide,
		$ \Delta (f_{J,1};x_2;x_1) =  \Delta (f_{J,1};x_2) = [\delta ; \infty [ = 
		\Delta (f;x_2) $.
		So, the designated center for $ f_{J,1} $ is $(x_1^{a_1},x_2^{a_2})$ as well. 	
 \end{proof}

\section{Rees algebras, weighted blow-ups, and proper transform}\label{Sec:weighted}\label{Sec:transform}

Now that we have a candidate for the designated center, let us recall the concept of a stack-theoretic weighted blow-up. A general reference for them is \cite{Quek-Rydh}.
Weighted blow-ups are very closely connected to W{\l}odarczyk's notion of cobordant blow-ups \cite{Wlodarczyk-Rees}, 
where a torus action is taken into account instead of taking the stack-theoretic quotient.

\begin{Def}[{\cite[Section~2.4]{ATW_weighted}}] 
		\label{Def:reduced_center}
		For $J = (x_1^{a_1},x_2^{a_2})$,
		the associated \emph{reduced center} is 
		$\bar J := (x_1^{1/w_1}, x_2^{1/w_2})$, 
		with 
		$w_1,w_2 \in \IZ_+ $ 
		being the unique pair of relatively prime integers 
		such that 
		\[ 
		\ell \cdot \left(\frac1{w_1},\frac1{w_2}\right) \, = \,  (a_1,a_2)
		\] 
		for some $\ell \in \IZ_+$. 
		\end{Def}
		To calculate $(w_1,w_2)$, recall that $a_1 \in \IZ_+$, and write $a_2 = \frac{b}{c}$ for some relatively prime integers $b,c \in \IZ_+$. Setting $k:=\gcd(a_1,b)$, write $a_1 = ka_1'$ and $b = kb'$. Then $\ell = ka_1'b'$, $w_1 = b'$ and $w_2 = a_1'c$.

We introduce the weighted blow-up in terms of extended Rees algebras \cite{Rees}.
For $J = (x_1^{a_1},x_2^{a_2})$ with reduced center $\bar J = (x_1^{1/w_1}, x_2^{1/w_2})$
and $ n \in \IZ $,
set
$ 
\cI_n := \{g \in \cO \mid  v_{\bar J}(g) \geq n\}. 
$
Note that $\cI_n = \cO$ for $n\leq 0$. 
The \emph{extended Rees algebra associated to $\bar J$} is the $ \cO $-algebra
\begin{equation}\label{eq:extended-rees-algebra}
	\tilde R \, =  \, \bigoplus_{n\in \IZ} \cI_n. 
\end{equation}

This local construction becomes global  by setting $ \cI_n = \cO_S $ away from $ q $.

\begin{Rk}\label{Rk:weighted-blow-ups}
		While the center of a blow-up on a scheme $\cW$ is an ideal $\cI \subset \cO_\cW$, the center of a weighted blow-up on $\cW$ is in general an extended Rees algebra on $\cW$. 
		\begin{enumerate}
			\item Firstly, a Rees algebra on $\cW$ is a descending sequence of ideals $\cI_\bullet = (\cO_\cW = \cI_0 \supset \cI_1 \supset \cI_2 \supset \dotsb)$ so that $\bigoplus_{n \geq 0}{\cI_n}$ forms a finitely generated, graded $\cO_\cW$-algebra. (See \cite[Definition 3.1.1]{Quek-Rydh}.) In particular, for $n > 1$, $V(\cI_n)$ are infinitesimal thickenings of $V(\cI_1)$, i.e. $\sqrt{\cI_1} = \sqrt{\cI_n}$. 
			In addition, a Rees algebra $\cI_\bullet$ is \emph{integrally closed} if $\bigoplus_{n \geq 0}{\cI_n}$ is integrally closed in the graded algebra $\bigoplus_{n \geq 0}{\cO_\cW}$.
            
			\item Recall that an ideal $\cI \subset \cO_\cW$ gives rise to a Rees algebra $\cI_\bullet$ where $\cI_n := \cI^n$ for every $n \geq 0$. In this way, usual blow-ups are special cases of weighted blow-ups.

            \item One can extend a Rees algebra in negative degrees by setting $\cI_n = \cO_\cW $ for $n < 0$. In this way, one associates to a Rees algebra an \emph{extended Rees algebra}, see \cite[Definition 3.1.3]{Quek-Rydh}. Doing so makes it convenient to write down a presentation for the $\cO_\cW $-algebra $\bigoplus_{n \in \IZ}{\cI_n}$. See \cite[Lemma 5.2.1, Proposition 5.2.2, and Corollary 5.2.5]{Quek-Rydh} for more details.

            \item 
            Let $\cW$ be a regular scheme. We call $\cI_\bullet$ a \emph{center} if at every point $p \in \cW$, there exist a regular system of parameters $x_1,x_2,\dotsc,x_d$ at $p$, as well as positive rational numbers $a_1 \leq a_2 \leq \dotsb \leq a_k$ for some $1 \leq k \leq d$, such that locally around $p$, we have $$\cI_n \; = \; \{g \in \cO_\cW \mid \nu(g) \geq n\} \qquad \text{for every $n \in \IZ$}$$ where $\nu$ is the valuation over $\cW$ with center $p$ defined by \[
                \nu(x_i) = \begin{cases}
                    \frac{1}{a_i} \qquad &\text{if $1 \leq i \leq k$} \\
                    0 \qquad &\text{if $k < i \leq d$}.
                \end{cases}
            \]
            In that case, we write $\cI_\bullet = (x_1^{a_1},x_2^{a_2},\dotsc,x_k^{a_k})$, and define the invariant of $\cI_\bullet$ at $p$ as $(a_1,a_2,\dotsc,a_k)$. Note that this is well-defined. 
            In the language of \cite{ATW_weighted} and \cite[Section 2.2]{MingHao}, $\cI_\bullet$ is the integrally closed Rees algebra on $\cW$ corresponding to the valuative $\IQ$-ideal $\gamma$ on $\cW$ whose stalk at a valuation $\nu$ over $\cW$ is $\min\{a_i\nu(x_i) \mid 1 \leq i \leq k\}$.
            \item 
            As with the invariant introduced in Definition~\ref{Def:invariant}, we order invariants of centers by the lexicographical ordering.
		\end{enumerate}
\end{Rk}

\begin{Ex}\label{Ex:cusp}
		If $J = (x_1^2,x_2^3)$, then the associated reduced center is $\bar J = (x_1^{1/3},x_2^{1/2})$. This means $x_1$ is given weight $3$ and $x_2$ is given weight $2$. The associated extended Rees algebra $\cI_\bullet$ is given by $\cI_n = (x_1^ax_2^b \mid 3a+2b \geq n)$ for all $n \in \IZ$. Starting from $n=1$, the first few terms are $\cI_1=(x_1,x_2)$, $\cI_2 = (x_1,x_2)$, $\cI_3 = (x_1,x_2^2)$, $\cI_4 = (x_1^2,x_2^2,x_1x_2)$ and $\cI_5 = (x_1^2,x_2^3,x_1x_2)$.
\end{Ex}

As indicated in Part (3) of Remark~\ref{Rk:weighted-blow-ups}, it is a consequence of \cite[Proposition~5.2.2]{Quek-Rydh} that $\tilde R$ has the following local presentation:
\begin{equation}
	\label{eq:presentation}
	\tilde R \, = \, \cO[ s, x_1', x_2' ] / (x_1 - s^{w_1} x_1',x_2 -s^{w_2} x_2'),
\end{equation}
for 
$ s \in \cI_{-1} $ 
the element $ 1 \in \cO $, and 
$ x_i' \in \cI_{w_i} $
corresponding to
$ x_i $ in degree $ w_i $, $ i \in \{1, 2\} $.
Here,
$s$ is the exceptional variable,
defining the exceptional divisor once 
the blow-up
is defined.
Moreover,
$x_i'$ is called the transformed variable corresponding to $x_i$, for $ i \in \{ 1, 2 \} $.

\begin{Constr}[{Degenerating to the weighted normal cone, see \cite{Wlodarczyk-Rees, Quek-Rydh}}]
	\label{Constr:Degen}
	Let 
	$B = \Spec_S(\tilde R)$ 
	be 
	the spectrum relative to the scheme $ S  $ of the quasi-coherent sheaf of $ \cO_S $-algebras $ \tilde R $.
	In addition to the
	structure morphism  $B \to S$, we have
	a morphism  
	$\varphi \colon B \to \IA_k^1 = \Spec k[s]$, which exhibits
	 $ B $ as
	the degeneration of $S$ to the weighted normal cone of
	$\bar J  $. Using Equation~\eqref{eq:presentation}, we have: 
\begin{itemize}
	\item 
	The fiber of $ \varphi $ over $s=1$ is isomorphic to $S$ 
	since
	$x_1=x_1', x_2=x_2'$. 
	
	\item 
	The fiber of $ \varphi $ over $s=0$ lies over the point $(x_1,x_2) = (0,0)$
and it is isomorphic to $\Spec k[x_1',x_2']$ 
	with the natural action of $\IG_m$ with weights $w_i$ on $x_i'$, making it the weighted normal cone 
	at $ (0,0) \in S $.
\end{itemize}

Observe that we have the following action of $\IG_m = \Spec k[t,t^{-1}]$
on $B$
\begin{equation}
	\label{eq:action}
	t \cdot (s,x_1',x_2') \quad = \quad (t^{-1} s, t^{w_1} x_1', t^{w_2} x_2') 
\end{equation} 
stabilizing
$V(\oplus_{n>0} \cI_n) = V(x_1', x_2')$.
\end{Constr}

\begin{Def}[Weighted blow-up]\label{Def:weighted}
	Let $ B = \Spec_S(\tilde R)$ 
	 be as in Construction~\ref{Constr:Degen}.
	The {\em weighted blow-up of $S$ along $\bar J$} is defined as the Artin stack
	\[ 
		\pi \colon \operatorname{Bl}_{\bar J} (S) := [B_+/\IG_m] \to S,
	\]
	where 
	$ B_+ := B \smallsetminus V(x_1', x_2') $, 
		with natural $\IG_m$-action determined by \eqref{eq:action} and
	the morphism $ \pi $ is induced by the structure morphism $ B \to S $.
\end{Def}

Observe that $ \operatorname{Bl}_{\bar J} (S)  $ is covered by two charts, namely where $ x_1' $ (resp.~$ x_2' $) is invertible in which case we also write $ x_1' \neq 0 $ (resp.~$ x_2' \neq 0 $).

\medskip

By \cite[Lemma~1.1.2 and Remark~3.2.10]{Quek-Rydh} the weighted blow-up fulfills the following properties.
Note that this proves Theorem~\ref{Thm:Main1}\eqref{It:birational}.

\begin{Lem}\label{Lem:weighted}
	Let 
	$ \pi \colon \operatorname{Bl}_{\bar J} (S) = [B_+/\IG_m] \to S $	be the weighted blow-up of $S$ along the reduced center $\bar J = (x_1^{1/w_1}, x_2^{1/w_2})$
	coming from $ J = (x_1^{a_1},x_2^{a_2}) $.
	Then:
	\begin{enumerate}
		\item 
		All stabilizers of the Artin stack  
		$ \operatorname{Bl}_{\bar J} (S) $ are 
		naturally subgroups of $\IG_m$.
		In particular, 
		$ \operatorname{Bl}_{\bar J} (S) $ is a {\em tame} Artin stack.
		
		\item 
		The dimension of the quotient 
		$ \operatorname{Bl}_{\bar J} (S) $ is two
		and 
		the morphism
		$ \pi $ is an isomorphism outside the exceptional divisor, which is defined on the smooth covering of $ B_+ $ by $ s = 0 $.
	\end{enumerate}

	In particular, the weighted blow-up $ \pi $ is birational.  
	
	\begin{enumerate}
		\item[(3)]
		At the point $s = x_1'=0$ the stabilizer group is $\Gmu_{w_2}$; 
		while at the point $s= x_2'=0$ the stabilizer group is $\Gmu_{w_1}$.
		Here, $ \Gmu_m $ denotes the group of $ m $-th roots of unity for $ m \in \IZ_+ $.
	\end{enumerate}
\end{Lem}

Let us next have a look at how an element $ f\in \cO $ transforms along a weighted blow-up.
Recall that we have given a singular point $ q $ on a
hypersurface $ C = V(f) $ 
in a regular $ 2 $-dimensional scheme which is essentially of finite type over a field $ k $.

Consider an expansion of $ f $ as in Equation~\eqref{eq:expansion}, 
$ f = f_{J,1}(x_1,x_2) + \widetilde h(x_1,x_2) $.
Recall that $ v_{\bar J} (f) = v_{\bar J} (f_{J,1}) = \ell < v_{\bar J} (\widetilde h) $ where $ \ell \in \IZ_+ $ is the integer introduced in Definition~\ref{Def:reduced_center}.
Using the presentation of $ \tilde R $ given in Equation~\eqref{eq:presentation},
the pullback of $ f $ to $ B$ is
\[
	f(x_1,x_2) 
	\ = \  f(s^{w_1} x_1',  s^{w_2} x_2') 
	\ = \  s^\ell f_{J,1}(x_1',x_2') + s^{\ell+1} \widetilde h' (s, x_1',x_2') 
\]
where $ \widetilde h' (s, x_1',x_2') $ is defined through
	$  \widetilde h(s^{w_1} x_1',  s^{w_2} x_2') :=
	s^{\ell+1}\widetilde h' (s, x_1',x_2') $.

\begin{Def}\label{Def:proper}
	The \emph{proper transform of $ f $ 
	on $B$} is defined as 
	the element $ f' $ that we obtain after factoring $ s^\ell $ from the pullback of $ f $ to $ B $,
	\[ 
		f'(s, x_1',x_2') 
		\ := \ 
		s^{-\ell} f(s^{w_1} x_1',  s^{w_2} x_2')  
		\ = \ 
		f_{J,1}(x_1',x_2') + s \, \widetilde h' (s, x_1',x_2').
	\]
\end{Def}

\begin{Ex}
		Consider the basic example of the cuspidal curve $f = x_1^2+x_2^3$ on the affine plane $\IA^2_K$. By Theorem~\ref{Thm:J_unique}, our designated center is $J = (x_1^2,x_2^3)$. As explained in Definition~\ref{Def:reduced_center} and Equation~\eqref{eq:extended-rees-algebra}, this leads to the weighted blow-up of $\IA^2_K$ along $\bar J = (x_1^{1/3},x_2^{1/2})$ (see also Example~\ref{Ex:cusp}). By Equation~\eqref{eq:presentation}, the coordinates $x_1$ and $x_2$ transform to $x_1 = s^3 x_1'$ and $x_2 = s^2 x_2'$, where $s$ is the exceptional variable. Notice that $f = f_{J,1}$. As a consequence, $f$ transforms to $s^6 (x_1'^2 + x_2'^3)$, so the proper transform of $f$ under the weighted blow-up is $f' = x_1'^2+x_2'^3$, which has no singularities since we removed the locus $V(\bigoplus_{n > 0}\cI_n) = V(x_1',x_2')$.
\end{Ex}

\begin{Ex}\label{Ex:one} 
	Consider again  Example~\ref{Ex:1_2_3}\eqref{Ex:2_localize},
	and assume for simplicity that $p\nmid m$. We have determined $\nu = 2p$ and $\delta = m/p$, giving $(a_1,a_2) = (2p,2m)$ and, keeping $(x_1,x_2) = (y,x)$ we have $J = \left(y^{2p}, x^{2m}\right)$. We now write $ f = f_{J,1}(y,x) + \widetilde h(y,x) $ with
	 \[ 
		 	f_{J,1}(y,x) \ = \ y^{2p} + x^m y^p z^{m+p} + x^{2m} z^{p+1} \qquad\text{and}\qquad \widetilde h(y,x) \ = \ x^{2m-1} y.
		 \]
		  Since $p\nmid m$ we have $\ell = 2mp$. We rescale to obtain $\bar J = \left(y^{1/m}, x^{1/p}\right)$, so that $x = s^px', y=s^my',$ giving
	 \[ 
		 	f'(y',x') \quad = \quad y'^{2p} + x'^m y'^p z^{m+p} + x'^{2m} z^{p+1} \quad + \quad  s^{m-p}x'^{2m-1} y'.
		 \]
		 The argument in the Section \ref{Sec:Proofs} below  guarantees  that its order along the exceptional locus $s=0$ is $<2p$. This follows from the fact that the  term $y'^{2p} + x'^m y'^p z^{m+p} + x'^{2m} z^{p+1}$ is not of the form $g^{2p}$
		and that $ x' $ or $ y' $ is invertible 
		 	(and also remember that $ z $ is invertible). 
\end{Ex}

\section{Proofs}
\label{Sec:Proofs}

This section starts by proving Theorem~\ref{Thm:Main1}\eqref{It:order-drops}  and Theorem~\ref{Thm:Main2}. Section~\ref{Sec:destack} is devoted to destackification.
We shall use Corollary~\ref{Cor:leading-term} to deduce that the invariant on any fiber of $ \varphi|_{B_+} \colon B_+ \to \IA_\K^1$ (see Construction~\ref{Constr:Degen}) drops strictly. In fact, we will show the stronger statement that the order decreases strictly after the designated weighted blow-up.

\begin{proof}[The logarithmic order drops: Proof of Theorem \ref{Thm:Main1}\eqref{It:order-drops}]
Since the logarithmic order with respect to the exceptional divisor $\{s=0\}$  
is functorial for smooth morphisms, showing that the logarithmic order drops on $B_+$ implies that it drops on the quotient $\Bl_{\bar J}(S) = [B_+/\IG_m]$. 
Note that $\nu_{\mm}(f) \leq \nu^{\log}_{\mm}(f)$, where the logarithmic order $\nu^{\log}_\mm (f)$ 
is defined: it is the order where $s\neq 0$ and otherwise it is the order of the restriction to $  V(s)  $, namely the order of $f_{J,1}(x_1',x_2')$.

When $s\neq 0$ we simply are computing the order of $f$ away from $ q = V(x_1, x_2) $. 
We note that near $ q $ this order is $<\nu=a_1$ since $f$ is not a pure $\nu$-th power.

When $s=0$ we are computing the order of $f_{J,1}(x_1',x_2')$ away from $ V(x_1',x_2') $. 
But once again, $f_{J,1}(x_1',x_2') = x_1'^{\nu} + \cdots$ is not a pure $\nu$-th power, see Corollary \ref{Cor:leading-term}. 
Thus its order of vanishing at any point on this locus is $<\nu=a_1$, as needed.
\end{proof}

\medskip

Using this result, we can now prove its equivariant generalization.

\begin{proof}[Proof of Theorem~\ref{Thm:Main2}]
Consider the localizations of $\cZ \subset \cW$ at the generic point $q:=\eta_\cQ$ of $\cQ$, namely $C := \Spec(\cO_{\cZ,q})$ and $S := \Spec(\cO_{\cW,q})$. Then $S$ is a 2-dimensional regular scheme that is essentially of finite type over $k$, and $C$ is cut out by a local equation $\{f=0\}$ in $S$. Note that the residue field of the closed point $q \in S$ is the function field $K(\cQ)$ 
and hence it is not necessarily perfect even if $ k $ was assumed to be perfect.

We apply part (1) of Theorem~\ref{Thm:Main1} to $C \subset S$ at $q$. There is a unique center $J=(x_1^{a_{1}}, x_2^{a_{2}})$ of maximal invariant $(a_{1},a_{2})$ that is admissible for $C \subset S$ at $q$, where $a_{2} = a_{1}\delta$, with $\delta$ the vertex of the characteristic polyhedron of $f$. 
Thinking of $J$ as a Rees algebra (see Remark~\ref{Rk:weighted-blow-ups}(4)), let $q \in U \subset \cQ$ be the largest open subset on which $J = (x_1^{a_{1}}, x_2^{a_{2}})$ extends.
In what follows, we show $U = \cQ$: 
\begin{enumerate}
    \item[(i)] Letting $k^\sep$ denote the separable closure of $k$, we claim $U$ is stable under $T(k^\sep)$. Indeed, $J$ remains the unique center of maximal invariant $(a_1,a_2)$ that is admissible for $C \times_{\Spec(k)} \Spec(k^{\sep})  \subset S \times_{\Spec(k)} \Spec(k^{\sep}) $ 
    at $q$. On the other hand, for every $t \in T(k^\sep)$, $t \cdot J$ is also a center on $S \times_{\Spec(k)} \Spec(k^{\sep}) $ that satisfies the same properties, so by the uniqueness of $J$, $t \cdot J = J$. Thus, $J = t \cdot J$ extends to $t \cdot U \subset \cQ$, so $t \cdot U \subset U$.

    \item[(ii)]We claim that in fact $U$ is stable under $T$. To do so, we instead show that $F = \cQ \smallsetminus U$ is stable under $T$, or equivalently, that the preimage of $F$ under the $T$-action $\alpha \colon T \times \cQ \to \cQ$ contains $T \times F$. By (i), $\alpha^{-1}(F) \supset \{t\} \times F$ for every $t \in T(k^\sep)$. Since $T$ is integral and geometrically reduced, the points in $T(k^\sep)$ are dense in $T$. Thus, after taking closures, we deduce $\alpha^{-1}(F) \supset T \times F$ as desired.
 
\end{enumerate}
Thus, $U \subset \cQ$ are both non-empty $T$-orbits, so $U = \cQ$. Hence, $J$ extends to $\cQ$, and moreover by (i), that extension to $\cQ$ is $T$-invariant.

Next, we claim that $J=(x_1^{a_{1}}, x_2^{a_{2}})$ is the unique $T$-invariant center of maximal invariant that is admissible for $\cZ \subset \cW$ along $\cQ$.
Note that a $ T $-invariant center is necessarily defined by at most 2 local parameters at all points.
To see the claim, any other such center $J'$ necessarily of invariant $(a_1,a_2)$ would localize to a center of maximal invariant that is admissible for $C \subset S$ at $q$. By the uniqueness of $J$ in Theorem~\ref{Thm:Main1}, $J$ agrees with $J'$ at $q$, and hence on a non-empty open subset of $\cQ$. Since $\cQ$ is an orbit under the $T$-action on $\cW$, and both $J$ and $J'$ are $T$-invariant, we deduce that $J = J'$. This completes the proof of part (1) of Theorem~\ref{Thm:Main2}.

We now prove the remaining parts of Theorem~\ref{Thm:Main2}. Following the notation introduced in Section~\ref{Sec:weighted}, consider the weighted blow-up of $[B_+/\Gm] \to \cW$ along $\bar J = (x_1^{1/w_1},x_2^{1/w_2})$. Since $B_+ \subset B = \Spec(\tilde{R})$ with $\cO_\cW \subset \tilde{R} \subset \cO_\cW[t,t^{-1}]$, the co-action $\cO_\cW \to \cO_\cW \otimes \cO_T$ naturally lifts to a co-action $\tilde{R} \to \tilde{R} \otimes \cO_T$ by sending $t \mapsto t \otimes 1$. In this way, the $T$-action on $\cW$ lifts to an action of $T' = T \times \Gm$ on $B_+$, inducing a morphism $[B_+/T'] \to [\cW/T]$ fitting in a cartesian diagram: \[
    \begin{tikzcd}
        {[B_+/\Gm]} \arrow[to=1-2] \arrow[to=2-1] & {[B_+/T']} \arrow[to=2-2] \\
        \cW \arrow[to=2-2]& {[\cW/T]}
    \end{tikzcd}
\]
Since $\cW \to [\cW/T]$ is a smooth cover, and $[B_+/\Gm] \to \cW$ is proper and birational by Lemma~\ref{Lem:weighted}, $[B_+/T'] \to [\cW/T]$ is also proper and birational. In order to show $[B_+/T']$ has finite stabilizers, it suffices to show its relative inertia over $[\cW/T]$ is finite, since by hypothesis $[\cW/T]$ has finite stabilizers. Indeed, that is the case, as the relative inertia of $[B_+/T']$ over $[\cW/T]$ is finite after base change to the smooth cover $\cW \to [\cW/T]$.

Finally, by Theorem~\ref{Thm:Main1}(3), the logarithmic order of $\cZ' \subset \cW' := B_+$ at every point lying over $q$ is strictly smaller than $a_1$. Since every point lying over $\cQ$ is in the same $T'$-orbit as points lying over $q$, we deduce the same conclusion for points in $\cZ'$ lying over $\cQ$.
\end{proof}

\subsection{From smooth tame stacks to smooth schemes}\label{Sec:destack}

A stack-theoretic resolution of singularities suffices for some applications, but not for others. It is therefore important to transform a stack-theoretic resolution to a scheme-theoretic resolution.  It is particularly desirable to keep such a process concrete and combinatorial, not requiring the depths of resolution of singularities.

One notes that the coarse moduli space $S$ of a smooth tame Artin stack $\cS$ with finite diagonal has toroidal singularities; in our case they are, moreover, quotient singularities by abelian group-schemes. There are several ways to resolve such singularities in a concrete and combinatorial way. We recall one here.

To keep things simple we concentrate on the case at hand of a smooth hypersurface in a smooth ambient stack with trivial generic stabilizers. 
In particular they are of finite type over a field.

\begin{Thm}[{Bergh \cite[Theorem 1.1]{Bergh}, see also Rosenblad \cite{Rosenblad}}]\label{Th:destack} Let $\cC_0 \subset \cS_0$ be a smooth hypersurface in a smooth, tame Artin stack with finite diagonal and abelian inertia, having coarse moduli spaces $C_0\subset S_0$. Let $\cU \subset \cS_0$ be the open locus where inertia is trivial, which we assume dense.

There is  a modification $\cS_1 \to \cS_0$, so that, writing  $\cC_1 \subset \cS_1$ for the proper transform of $\cC_0$, 
\begin{itemize}
\item $\cS_1$ and $\cC_1$ are smooth,
\item $\cS_1 \to \cS_0$ is an isomorphism along $\cU$, and
\item  The coarse moduli spaces $S_1$ of $\cS_1$ and $C_1$ of $\cC_1$ are also smooth, and are projective over $S_0$ and $C_0$.
\end{itemize}
In particular $C_1 \to C_0$ and $S_1 \to S_0$ are scheme-theoretic resolutions of singularities.
\end{Thm}

We emphasize that Bergh's theorem results in an \emph{algorithm}, which is indeed concrete and combinatorial. It uses only standard, unweighted blow-ups and root constructions. It has been implemented by a group of undergraduates \cite{destackification1,destackification2}. The original algorithm in Bergh's paper is highly inefficient, but Rosenblad's paper \cite{Rosenblad} makes it significantly more efficient using weighted blow-ups.

As a consequence, we obtain a concrete, combinatorial, and computer-implemented upgrade of Theorem \ref{Thm:resolution} to the well-known scheme-theoretic resolution of plane curves:

\begin{Cor}[{See \cite[Chapter 3]{Cutkosky}, \cite[Chapter 1]{Kollar}}] \label{Cor:coarse-resolution}
Let $C \subset S$ be a hypersurface in a regular $ 2 $-dimensional scheme $S$ which is essentially of finite type over a field $ k $.
 Then there is a projective morphism $S_1 \to S$, with $S_1$ smooth, so that the proper transform $C_1$ of $C$ is also smooth.
 	\end{Cor}

The following diagram indicates the processes involved:

$$\xymatrix{
\cC_1 \subset \cS_1 \ar[rrrrr]^{\text{Destackification,}}_{\text{Theorem \ref{Th:destack}}} \ar[dd]|-{\text{Coarse moduli space}} &&&&&
  \cC_0\subset \cS_0\ar@/^2em/[rrrrrrdd]|-{\text{{\qquad \qquad \    Weighted resolution, Theorem  \ref{Thm:resolution}}}}	\ar[dd]|-{\text{coarse moduli space}}  \\ \\
  C_1 \subset S_1 \ar[rrrrr]^{\text{Resolution of quotient singularities}}_{\text{resulting from Theorem \ref{Th:destack}}} \ar@/_3em/[rrrrrrrrrrr]_{\text{Resolution, Corollary \ref{Cor:coarse-resolution}}}&&&&&   C_0\subset S_0\ar[rrrrrr]^{\text{Resolution up to quotient singularities}}_{\text{resulting from Theorem \ref{Thm:resolution}}} &&&&&& C \subset S\\ \\
  }$$

\section{Comparison of approaches}\label{Sec:comparison}
We end this paper by comparing different methods of resolving plane curve singularities for explicit
examples.

\begin{Ex}
	\label{Ex:compare1}
	We consider the plane curve given by the polynomial 
	\[  
		f(x,y)  \, = \, y^4 + x^3  \ \in \  K[x,y]   
	\] 
	from different perspectives, where $ K $ is a field. 
	\addtocounter{subsection}{1}
	\subsubsection{Newton's method} First, suppose $ \car(K) = 0 $.
	A first method to resolve the singularities is {\em Newton's algorithm} to determine a Puiseux expansion for the given \emph{unibranch} curve.
	The algorithm is described in \cite[Section~2.1]{Cutkosky} in detail
	and involves transforms similar to those of the weighted blow-up. 
	For the given example, a possible Puiseux expansion is $ x = -y^{4/3} $.
	This defines an inclusion 
	$   R := K [[x,y]]/ (f(x,y)) \to K [[t]] $ via $ (x,y) \mapsto (-t^4 , t^3 ) $
	for which source and target have the same quotient field (see \cite[beginning of Section~2.5]{Cutkosky}).
	In particular, $ \Spec (K [[t]]) \to \Spec (R) $ is a local, nonembedded desingularization. 
	Let us point out that this method fails to work in positive characteristic in general, see \cite[Remark~2.2 and Exercise~2.4.2]{Cutkosky}.
	In general, Puiseux expansions do not exist in positive characteristic and one has to work with the more technical notion of Hamburger-Noether expansions, see \cite[Charpter~II]{Campillo1} or \cite{Campillo2}. 
	A comparison with that work would be interesting but would go too far afield.

	\subsubsection{Classical blow-ups} Next, let us determine a resolution via usual, non-weighted, blow-ups. 
	Here, $ K $ can be any field.
	First we have to blow up the origin. 
	We write the 
	usual 
	blow-up in terms of its degeneration to the normal cone, see Construction \ref{Constr:Degen}. (To obtain the classical charts, set $x'=1$ with $s=x$ for one chart, and set $y'=1$ with $ s=y$ for the other.)
	By \eqref{eq:presentation}, the coordinate transform as 
	$ (x,y) = ( s x', s y' ) $, 
	where $ s $ is the exceptional variable,
	and we have to consider points away from $ V(x',y') $.
	The proper transform of $ f $ is 
	$ f' = s y'^4 + x'^3 $.
	We observe that $ V(f') $ is already regular away from $V(x',y') $.
	Of course, this is an accident, as many more steps are needed in most other examples, see Example~\ref{Exbu-wbu}. We do not pursue logarithmic resolution here, though it is part of the process in  \ref{Ex:toric}  and happens automatically in this example in \ref{Ex:weighted} , \ref{Ex:multi-weighted}.
	
	\subsubsection{Toric methods}\label{Ex:toric} Since the original $ f $ is a binomial, the singularities of $ V(f) $ can also be resolved using {\em toric methods}.
	Still, $ K $ can be any field.
	(This is also the resolution process following \cite{ACampoOka} for this example.)
	For this, one has to describe a subdivision of the fan determined by the positive orthant $ \IR_{\geq 0}^2 $. 
	The Newton polyhedron, 
	which is determined by the exponents of the appearing monomials, 
	has a single compact edge with vertices $ (0,4) $ and $ (3,0) $.
	The normal vectors of the Newton polyhedron $ \binom{1}{0}, \binom{4}{3}, \binom{0}{1} $ determine a first subdivison of $ \IR_{\geq 0}^2 $.
	This already defines a toric modification, 
	but since there are neighboring rays whose determinant is not invertible in $ \IZ $,
	the resulting ambient space is a singular toric variety.
	Hence, a further refinement is needed. 
	This leads to the fan whose set of rays is   
	\[ 
	\left\{ 
	\binom{1}{0}, \ \binom{2}{1}, \ \binom{3}{2}, \ \binom{4}{3}, \ \binom{1}{1}, \ \binom{0}{1}
	\right\}
	.
	\]
	The corresponding fan looks as follows
	where the minimal generators of the rays are marked as thick vectors:
	\[
		\begin{tikzpicture}[scale=0.9]
			
			\path[pattern=north east lines, pattern color=black!20!white,dashed] 
			(4.5,0)--(0,0)--(4.5,2.25);
			
			\path[pattern=horizontal lines, pattern color=black!20!white,dashed] 
			(4.5,2.25)--(0,0)--(4.5,3);
			
			\path[pattern=vertical lines, pattern color=black!20!white,dashed] 
			(4.5,3)--(0,0)--(4.5,3.375);
			
			\path[pattern=north west lines, pattern color=black!20!white,dashed] 
			(4.5,3.375)--(0,0)--(4.5,4.5);
			
			\path[pattern=dots, pattern color=black!20!white,dashed] 
			(4.5,4.5)--(0,0)--(0,4.5);
			
			\foreach \i in {1,...,4} {
				\draw [lightgray] (\i,0) -- (\i,4.5);
				\draw [lightgray] (0,\i) -- (4.5,\i);
			}

			\draw (4.5,0)--(0,0)--(0,4.5);
			\draw (4.5,2.25)--(0,0)--(4.5,3);
			\draw (4.5,3.375)--(0,0)--(4.5,4.5);
			
			\draw[<->, ultra thick] (1,0)--(0,0)--(0,1);
			\draw[<->, ultra thick] (2,1)--(0,0)--(3,2);
			\draw[<->, ultra thick] (4,3)--(0,0)--(1,1);

		\end{tikzpicture} 
	\]
	The charts correspond to the five maximal cones. It can be verified on each of them that the strict transform is regular and either disjoint from or transvesral to the exceptional divisor.
	For example, for the cone $ \sigma =\left\langle  \binom{1}{1}, \binom{4}{3} \right\rangle$ with $\sigma^\vee = \left\langle   \binom{1}{-1}, \binom{-3}{4} \right\rangle,$
	the coordinates transform is $(x_\sigma,y_\sigma) = (xy^{-1},x^{-3}y^4)$ giving $(x,y) = ( x_\sigma^4 y_\sigma, \, x_\sigma^3 y_\sigma ), $
	and thus 
	\[  
		f \, = \, 
		y^4 + x^3 
		\, = \, 
		x_\sigma^{12} y_\sigma^{4} + x_\sigma^{12} y_\sigma^3
		\, = \, 
		 x_\sigma^{12} y_\sigma^3 ( y_\sigma +  1 )
		 .
	\]
	Note that the strict transform of $ f $ is $ y_\sigma + 1 $ and the exceptional divisors are $ V( x_\sigma) $ and $ V(y_\sigma)  $.

	In \cite{GGP}, a different perspective on the construction of the subdivision using the notion of lotuses is given. 
	Lotuses are combinatorial objects encoding much information on the curve singularity.
	Since a deeper discussion would lead too far, we only refer to the original article for details.

	\subsubsection{Weighted blow-ups}\label{Ex:weighted} The reduced center for the {\em weighted blow-up} of $ f = y^4 + x^3 $ is $ \bar{J} = (x^{1/w_1}, y^{1/w_2} )= (x^{1/4}, y^{1/3} ) $ in any characteristic.
	By \eqref{eq:presentation}, the coordinate transformation is 
	$ (x,y) = ( s^4 x', s^3 y' )$.
	This provides
	$ 
		f = s^{12} y'^4 + s^{12} x'^3 = s^{12} (y'^4 + x'^3)
	$. 
	The exceptional divisor is $ V (s)  $ and the proper transform $ f' = y'^4 + x'^3 $ is regular as we have to remove $ V (x',y') $.
	Therefore, the singularities are already resolved.
	The following picture shows the stacky fan, see \cite[Definition 2.4]{Toric1},
	corresponding to the weighted blow-up, with the stacky lattice $L\subset \IZ^2$ marked by solid red points:
	\[
	\begin{tikzpicture}[scale=0.6]		
		
		\foreach \i in {1,...,9} {
			\draw [lightgray] (\i,0) -- (\i,6.5);
		}
		\foreach \i in {1,...,6} {
			\draw [lightgray] (0,\i) -- (9.5,\i);
		}
		
		\draw[thick] (9.5,0)--(0,0)--(0,6.5);

		\draw (0,0)--(8.4,6.3);
		
		\foreach \i in {0,...,9}
		\foreach \j in {0,...,6} {
			\filldraw[black] (\i,\j) circle (1.3mm);
			\filldraw[white] (\i,\j) circle (1.1mm);
		}
		
		\filldraw[red] (4,3) circle (1.5mm);
		\filldraw[red] (8,6) circle (1.5mm);

		\foreach \i in {0,...,9} {
			\filldraw[red] (\i,0) circle (1.5mm);
		}
		
		\foreach \i in {5,...,9} {
			\filldraw[red] (\i,3) circle (1.5mm);
		}
		
		
		\foreach \i in {1,...,6} {
			\filldraw[red] (0,\i) circle (1.5mm);
		}

		\foreach \i in {4,...,6} {
			\filldraw[red] (4,\i) circle (1.5mm);
		}

		\draw[<->, ultra thick] (1,0)--(0,0)--(0,1);
		\draw[->, ultra thick] (0,0)--(4,3);

	\end{tikzpicture} 
	\]
	
	The underlying fan, with  red fillings ignored,  is the fan of the coarse moduli space. The coarse moduli space can be obtained directly by blowing up the ideal $(x^3,y^4)$ and normalizing,  see \cite[Definition 2.5 and Remark 3.11]{AQ}.  The red filled lattice points indicate, in each cone, the toric  chart of a toric abelian finite cover whose stack-theoretic quotient gives an open chart of the stack-theoretic blow-up.
	
	\subsubsection{Multi-weighted blow-ups}\label{Ex:multi-weighted} Finally, let us illustrate how \cite{AQS-plane} treats this singularity. 
	Assume that $ \car(K) = 3 $.
	Since $ \car (K) $ divides $ w_2 = 3 $, 
	the ambient space resulting from the weighted blow-up is an Artin stack, but not a Deligne--Mumford stack, as explained in the introduction of \cite{AQS-plane}.
	In loc.~cit., 
	a resolution via Deligne--Mumford stack is constructed 
	using multi-weighted blow-ups.
	Explicitly, for the given example, the fan determined by the normal vectors of the Newton polyhedron has to be refined by introducing the new ray $ \binom{\kappa}{1} $ 
	where $ \kappa := \left\lceil {w_1}/{w_2} \right\rceil = 2 $.  
	The coordinate transform for the corresponding multi-weighted blow-up is 
	\[
		(x,y)  \ =  \ (s_1^4 s_2^2  x'', \, s_1^3  s_2 y'')
	\]
	so that
	\[
		 f
		 \,  = \, 
		 y^4 + x^3 
		 \, = \, 
		 s_1^{12}  s_2^4 y''^4 + s_1^{12} s_2^6  x''^3
		 \, = \, 
		 s_1^{12}  s_2^4  (y''^4 + s_2^2 x''^3 ) .
	\]	
	The exceptional divisors are $ V (s_1)  $ and $ V (s_2)  $
	and the proper transform of $ f $ is $ f'' := y''^4 + s_2^2 x''^3  $.
	Here, we have to remove	$ V ( s_2 x'' , x'' y'' , s_1 y'' )  $ 
	-- for the details why this locus has to be removed, we refer to \cite[Definition 4.1, Remarks~4.2 and 4.3]{Toric1}
	or \cite[Definition~2.1, Remark~2.2]{AQ}. 
	In particular, we see that the proper transform $ f'' $ is regular and transveral to the exceptional locus.
	The stacky fan corresponding to the multi-weighted blow-up looks as follows: 
	\[
\begin{tikzpicture}[scale=0.6]		
	
	\foreach \i in {1,...,9} {
		\draw [lightgray] (\i,0) -- (\i,6.5);
	}
	\foreach \i in {1,...,6} {
		\draw [lightgray] (0,\i) -- (9.5,\i);
	}
	
	\draw[thick] (9.5,0)--(0,0)--(0,6.5);
	
	\draw (9.4,4.7)--(0,0)--(8.4,6.3);
	
	\foreach \i in {0,...,9}
	\foreach \j in {0,...,6} {
		\filldraw[black] (\i,\j) circle (1.3mm);
		\filldraw[white] (\i,\j) circle (1.1mm);
	}
	
	\filldraw[red] (4,3) circle (1.5mm);
	\filldraw[red] (6,4) circle (1.5mm);
	\filldraw[red] (8,5) circle (1.5mm);
	\filldraw[red] (8,6) circle (1.5mm);

	\foreach \i in {0,...,9} {
		\filldraw[red] (\i,0) circle (1.5mm);
	}
	
	\foreach \i in {0,...,9} {
		\filldraw[red] (\i,0) circle (1.5mm);
	}
	\foreach \i in {2,...,9} {
		\filldraw[red] (\i,1) circle (1.5mm);
	}
	\foreach \i in {4,...,9} {
		\filldraw[red] (\i,2) circle (1.5mm);
	}
	\foreach \i in {6,...,9} {
		\filldraw[red] (\i,3) circle (1.5mm);
	}
	\foreach \i in {8,...,9} {
		\filldraw[red] (\i,4) circle (1.5mm);
	}

	\foreach \i in {1,...,6} {
		\filldraw[red] (0,\i) circle (1.5mm);
	}
	
	\foreach \i in {4,...,6} {
		\filldraw[red] (4,\i) circle (1.5mm);
	}

	\draw[<->, ultra thick] (1,0)--(0,0)--(0,1);
	\draw[<->, ultra thick] (2,1)--(0,0)--(4,3);

\end{tikzpicture} 
\]

\subsubsection{Destackifying the multi-weighted and weighted blow-ups}  The stacky fan of Section \ref{Ex:multi-weighted} above provides a relatively short example of destackification. We point out that destackification is often a lengthy operation, and Bergh's version \cite{Bergh} is longer than Rosenblad's \cite{Rosenblad}.
For toric surfaces one can show from first principles that the resulting resolution is at least as long as the one obtained by classical toric methods, and any toric resolution of the coarse moduli space is  the coarse moduli space of a destackification.

In this particular case, all methods lead to the fan of Section \ref{Ex:multi-weighted} given by the rays $\left\{ 
	\binom{1}{0},  \binom{2}{1},  \binom{3}{2},  \binom{4}{3},  \binom{1}{1},  \binom{0}{1}
	\right\}$, which is the minimal possible fan.
\begin{itemize}
\item 	
The cone 	$\left\langle 
	\binom{1}{0}, \binom{2}{1}\right\rangle$ is already regular, so no operations are needed.
\item  The cone $\left\langle 	\binom{2}{1},  \binom{4}{3}\right\rangle$ has determinant 2. One performs the standard blow-up of the stack at the barycentric generator $\binom{2}{1} + \binom{4}{3} = \binom{6}{4} = 2\binom{3}{2}.$ The coarse moduli space thus has the new generator $\binom{3}{2}$ and the subdivision of the cone is regular.
\item  The cone $\left\langle 	\binom{4}{3}, \binom{0}{1} \right\rangle$ has determinant 4. One performs the standard blow-up of the stack at the barycentric generator $\binom{4}{4} = 4\binom{1}{1}.$ The coarse moduli space thus has the new generator $\binom{1}{1}$ and the subdivision of the cone is regular.
\end{itemize}
Here is the resulting stacky fan, showing that both the stack and its coarse moduli space are regular:
	\[
\begin{tikzpicture}[scale=0.6]		
	
	\foreach \i in {1,...,9} {
		\draw [lightgray] (\i,0) -- (\i,6.5);
	}
	\foreach \i in {1,...,6} {
		\draw [lightgray] (0,\i) -- (9.5,\i);
	}
	
	\draw[thick] (9.5,0)--(0,0)--(0,6.5);
	
	\draw (9.4,4.7)--(0,0)--(8.4,6.3);
	\draw (9,6)--(0,0)--(6,6);
	
	\foreach \i in {0,...,9}
	\foreach \j in {0,...,6} {
		\filldraw[black] (\i,\j) circle (1.3mm);
		\filldraw[white] (\i,\j) circle (1.1mm);
	}
	
	\filldraw[red] (4,3) circle (1.5mm);
	\filldraw[red] (6,4) circle (1.5mm);
	\filldraw[red] (8,5) circle (1.5mm);
	\filldraw[red] (8,6) circle (1.5mm);

	\foreach \i in {0,...,9} {
		\filldraw[red] (\i,0) circle (1.5mm);
	}
	
	\foreach \i in {0,...,9} {
		\filldraw[red] (\i,0) circle (1.5mm);
	}
	\foreach \i in {2,...,9} {
		\filldraw[red] (\i,1) circle (1.5mm);
	}
	\foreach \i in {4,...,9} {
		\filldraw[red] (\i,2) circle (1.5mm);
	}
	\foreach \i in {6,...,9} {
		\filldraw[red] (\i,3) circle (1.5mm);
	}
	\foreach \i in {8,...,9} {
		\filldraw[red] (\i,4) circle (1.5mm);
	}
	
	
	\foreach \i in {1,...,6} {
		\filldraw[red] (0,\i) circle (1.5mm);
	}
	
	\foreach \i in {4,...,6} {
		\filldraw[red] (4,\i) circle (1.5mm);
	}
	
	
	\draw[<->, ultra thick] (1,0)--(0,0)--(0,1);
	\draw[<->, ultra thick] (2,1)--(0,0)--(4,3);
	\draw[<->,  thick, red] (6,4)--(0,0)--(4,4);	
	\draw[<->,  thick] (3,2)--(0,0)--(1,1);	
	
\end{tikzpicture} 
\]

If one were to apply Bergh's algorithm to the \emph{weighted} blow-up in Section \ref{Ex:weighted}, a redundant ray $\binom{5}{3}$ would appear right away. Rosenblad provides a more efficient algorithm. Start with a \emph{weighted} stacky blow-up subdividing the cone $\left\langle \binom{1}{0},\binom{4}{3}\right\rangle$ at $\binom{4}{3}+2\binom{1}{0} = \binom{6}{3} = 3\binom{2}{1}$, and then proceed very similarly to the procedure above, with a weighted blow-up at $3\binom{4}{3}+\binom{6}{3} = \binom{18}{12} = 6\binom{3}{2}$, again recovering the toric fan of Section \ref{Ex:toric}. The margins here are too narrow for the required picture.

\end{Ex}

\begin{Ex}
	\label{Exbu-wbu}  
		Let us briefly outline an example, where the resolution via non-weighted blow-ups requires significantly more steps than with weighted blow-ups.
		
		Analogous to Example~\ref{Ex:compare2},
		the plane curve singularity 
		\[  
			V( y^{21} + x^{34}) 
		\]
		is resolved by the weighted blow-up with reduced center $ \bar{J} = (x^{1/21}, y^{1/34}) $. 
		On the other hand, it can be verified that six non-weighted blow-ups are needed to resolve the singularities. 
		We leave the details of the computation as an exercise to the reader. 
		
		The difference in numbers of blow-ups can be arbitrarily large,
		as the following generalization of the example shows.
		For $ n \in \IZ_+ $, $ n \gg 0 $, 
		consider the curve singularity 
		\[ 
			V(y^{F_{n+2}} + x^{F_{n+3}}) 
		\] 
		for two consecutive Fibonacci numbers $ F_{n+2}, F_{n+3} $.  
		Again, a single weighted blow-up resolves the singularities.
		The reader may verify that the number of non-weighted blow-ups required to resolve the singularity is $ n $.
		We point out that this example is inspired by the work \cite{Fibonacci} of Pe Pereira and Popescu-Pampu.
\end{Ex}

\begin{Ex}
	\label{Ex:compare2}
	In order to compare the weighted blow-ups with the desingularization via Tschirnhausen resolution tower by A'Campo and Oka \cite{ACampoOka},
	we have to study a slightly more complicated example. 
	Consider  
	\[
		g(x,y) = (y^4 + x^3)^6 + x^{17} y^3 
	\]
	over any base field $ K $ of characteristic zero. 
	As explained in \cite[Example~4.9]{ACampoOka},
	the singularities are resolved using their method as follows:
	First make the toric modification resolving the singularities of $ V ( y^4 + x^3 ) $
	(as discussed in Example~\ref{Ex:compare1}).
	Then in one of the charts, 
	the strict transform of $ f $ has the form $ v_1^6 + u_1^5 + (\mbox{higher terms}) $
	as explained in \cite[Example~4.9]{ACampoOka}.
	Hence, another toric modification 
	corresponding to the fan with set of rays
	\begin{equation} \label{Eq:AOfan}
		\left\{ 
		\binom{1}{0}, \ \binom{2}{1}, \ \binom{3}{2}, \ \binom{4}{3}, \ \binom{5}{4}, \ \binom{6}{5}, \ \binom{1}{1}, \ \binom{0}{1}
		\right\}
	\end{equation}
	has to be performed in order to resolve the singularities. 
	
	For the resolution with weighted blow-ups, 
	we first blow up with reduced center $ \bar{J} = (x^{1/w_1}, y^{1/w_2} )= (x^{1/4}, y^{1/3} ) $, as in Example~\ref{Ex:compare1}.
	Applying $ (x,y) = ( s^4 x', s^3 y' )$,
	we get 
	\[ 
		g \ = \ 
		s^{72} (y'^4 + x'^3)^6 + s^{77} x'^{17} y'^3
		\ = \
		s^{72} \left(  (y'^4 + x'^3)^6 + s^5 x'^{17} y'^3 \right)
		. 
	\]
	The proper transform is $ g' = (y'^4 + x'^3)^6 + s^5 x'^{17} y'^3 $.
	Observe that the order of $ g' $ is at most 6 at every point away from $ V ( x ' , y' ) $, 
	while the order of $ g $ at the origin is 18 before the blow-up. 
	
	Let us briefly outline how this resolution process ends:
	If $ y'^4 + x'^3  $ is invertible, then $  V (g')  $ is regular and transveral to the exceptional divisor $ V  (s)  $.
	Hence, look at the points where we have $ y'^4 + x'^3 = 0 $.
	Since we are away from $ V ( x' , y')  $, we can  
	introduce the new local coordinate $ z := y'^4 + x'^3 $.
	Notice that both $ x' $ and $ y' $ are invertible on $ V (z)  $.  
	Therefore, the proper transform can be written as 
	$ g' = z^6 + \lambda s^5 $ where $ \lambda := x'^{17} y'^3 $ is invertible. 
	Analogous to Example~\ref{Ex:compare1},
	the weighted blow-up with reduced center $ \bar J' := (z^{1/5}, s^{1/6}) $ will resolve the singularities. The stacky fan in these variables has generators 	$
		\left\{ 
		\binom{1}{0},  \binom{6}{5},  \binom{0}{1}
		\right\}
	$ and a minimal destackification will have the fan \eqref{Eq:AOfan} of A'Campo and Oka.

	Finally, let us briefly outline Teissier's approach through {\em overweight deformations} \cite{Teissier,GoldinTeissier}.
	Here, the goal is to construct a suitable re-embedding into a possibly higher-dimensional ambient space 
	such that the singularities can be resolved with a single toric modification in the new ambient space. 
	For the given example such a re-embedding is given by introducing a new variable $ z $ fulfilling the relation $ z = y^4 + x^3 $ so that $ V (f) $ is isomorphic to
	\[
		\left\{ \ 
		\begin{array}{rcl}
			y^4 + x^3 - z  & = &  0 	,
			\\
			z^6 + x^{17} y^3 & = & 0. 
		\end{array}
		\right.
	\]
	The latter is an overweight deformation of the binomial variety $ V ( y^4 + x^3 , \, z^6 + x^{17} y^3 ) $
	which means that there is a weight function $ W $ such that the binomials are weighted homogeneous with respect to $ W $,
	while the remaining terms (i.e., $ z $ in the first equation) are of higher weight. 
	Indeed, this is fulfilled if we set $ W (x) = 4, W(y) = 3, W (z) = 77/6 $. 
	The desingularization of $ V(f) $ can then be deduced from the one of the binomial variety.	
	More background on how the re-embedding is constructed and how the resolution is obtained can be found in \cite{Teissier,GoldinTeissier} or in \cite{PCQO} where irreducible quasi-ordinary hypersurface singularities are investigated from the overweight perspective of Teissier. 
	The latter are a bigger class of singularities including irreducible plane curve singularities.
	We only point out that the construction of such a re-embedding  requires us to know the generators of the semi-group associated to the singularity (e.g., see \cite{GoldinTeissier} or \cite{FGM}),
	and the latter are closely connected to a Puiseux expansion. 
\end{Ex}

	\begin{Rk} For non-unibranch singularities, using either the method of Puiseux expansion or of \cite{GoldinTeissier}, one would typically first separate the branches and then find a Puiseux expansion for  each. The other methods do not require isolating branches, which get separated along the process.
	
	For instance, with weighted resolution, to resolve the singularity $V((x^3+y^4)(y^3+ x^5))$, which has branches with  distinct tangent cones, one starts with the center $(x^6,y^6)$ leading to a classical blow-up separating the branches and simplifying them simultaneously. 
	
	In contrast, resolving $V((x^3+y^4)(x^3+ y^5))$, with branches having identical tangent cones, the center $(x^6,y^8)$, with reduced center $(x^{1/4},y^{1/3})$, picks out the branch $V(x^3+y^4)$ with smaller exponent $4/3<5/3$ to resolve first.
	
	Note that the result of Goldin and Teissier \cite{GoldinTeissier} was extended from unibranch to non-unibranch curve singularities in \cite{FGM}.
	\end{Rk}

\end{document}